\documentclass{article}

\usepackage{amsmath,amsfonts,amsthm,amssymb,graphicx,tikz,tikz-cd}
\usepackage[all,2cell,ps]{xy}

\theoremstyle{plain}
\newtheorem{thm}{Theorem}[section]

\newtheorem{prop}[thm]{Proposition}

\theoremstyle{definition}

\newcommand{\Z}{\mathbb Z}
\newcommand{\R}{\mathbb R}
\newcommand{\Q}{\mathbb Q}

\newcommand{\C}{\mathbb C}

\newcommand{\sig}{\sigma}

\newcommand{\gam}{\gamma}

\newcommand{\Del}{\Delta}

\newcommand{\Lam}{\Lambda}
\newcommand{\Gam}{\Gamma}

\DeclareMathOperator{\PGL}{PGL}
\DeclareMathOperator{\GL}{GL}
\DeclareMathOperator{\Sp}{Sp}

\DeclareMathOperator{\Or}{O}
\DeclareMathOperator{\PO}{PO}

\newenvironment{pf}{\begin{proof}}{\end{proof}}

\usetikzlibrary{arrows}

\title{New nonlinear hyperbolic groups}
\author{Richard D. Canary \\ \small{University of Michigan} \\ \small{\textsf{canary@umich.edu}} \and Matthew Stover \\ \small{Temple University}\\ \small{\textsf{mstover@temple.edu}}
\and Konstantinos Tsouvalas \\ \small{University of Michigan} \\ \small{\textsf{tsouvkon@umich.edu}}}
\date{\today}

\begin{document}

\maketitle

%%%%%%%%%%%%%%%%%%%%
\begin{abstract}
We construct nonlinear hyperbolic groups which are large, torsion-free, one-ended, and admit a finite $K(\pi,1)$.
Our examples are built from superrigid cocompact rank one lattices via
amalgamated free products and HNN extensions.
\end{abstract}
%%%%%%%%%%%%%%%%%%%%

%%%%%%%%%%%%%%%%%%%%
\section{Introduction}\label{sec:Intro}
%%%%%%%%%%%%%%%%%%%%

In this note, we construct new examples of nonlinear hyperbolic groups. For us, a group is ``nonlinear'' if it does not admit a faithful representation into $\GL_n(F)$ for $F$ any field.
As with previous constructions, our groups are built from superrigid cocompact lattices in rank $1$ Lie groups. Previous examples
were quotients of such lattices, small cancellation theory was used 
to show that the quotients are hyperbolic, and superrigidity results were used to see that they are nonlinear (see M.\ Kapovich \cite[\S 8]{MKap}).
Our construction involves simple HNN extensions and free products with amalgamation, and one can prove that the resulting groups
are hyperbolic using the Bestvina--Feighn Combination Theorem \cite{BF}. Our examples are large (i.e., have finite
index subgroups that surject a free group of rank two), torsion-free, one-ended, and admit a finite $K(\pi,1)$.

%%%%%%%%%%%%%%%%%%%%
\begin{thm}\label{mainthm}
For any $n\ge 0$, there exist large, torsion-free, one-ended, nonlinear hyperbolic groups that admit a finite $K(\pi,1)$,
have first betti number $n$, and surject a  free group of rank $n$.
\end{thm}
%%%%%%%%%%%%%%%%%%%%

We present two related constructions, both of which begin with a cocompact torsion-free lattice $\Gam$ in $\Sp(m,1)$ (always with $m\ge 2$)
or $\mathrm{F}_4^{(-20)}$. As in M.\ Kapovich \cite{MKap}, our proofs rely crucially on Corlette's \cite{Corlette} and Gromov--Schoen's \cite{Gromov--Schoen} generalizations of the Margulis superrigidity theorem to lattices in these groups. 
In what follows, let $G$ be $\Sp(m,1)$ or $\mathrm{F}_4^{(-20)}$ and 
$X$ be the associated rank one symmetric space, i.e., quaternionic hyperbolic \hbox{$m$-space} or the Cayley hyperbolic plane.

In our first construction, we choose elements $\gamma_1$ and $\gamma_2$ of $\Gamma$ 
associated with primitive closed geodesics of different length in the locally symmetric space
$X/\Gamma$.
We consider the group $\Lambda_1$ obtained by taking the
HNN extension of $\Gamma$ such that the stable letter conjugates $\gamma_1$ to $\gamma_2$, i.e.,
\[
\Lambda_1=\langle \Gamma,t\ |\ t\gamma_1t^{-1}=\gamma_2\rangle.
\]
We use  superrigidity results  to show that if $\Lambda_1$ is linear,
then it admits a faithful representation $\rho$  into $\GL_n(\mathbb R)$ and there is  a totally geodesic embedding of $X$ into
the symmetric space $Y_n$ of $\GL_n(\R)$ which is equivariant with respect to the restriction
$\rho|_\Gamma$ of $\rho$ to $\Gamma$. Since the translation lengths of
$\rho(\gamma_1)$ and $\rho(\gamma_2)$ agree in $Y_n$ and $f$ is totally geodesic, the translation lengths
of $\gamma_1$ and $\gamma_2$ on $X$ agree, which gives a contradiction.
It follows that $\Lambda_1$ is nonlinear.
The Bestvina--Feighn combination theorem \cite{BF} implies that $\Lambda_1$ is hyperbolic, 
and it is clear that $\Lambda_1$ has first betti number $1$, 
has the same cohomological dimension as
$\Gamma$, admits a finite
$K(\pi,1)$, and is torsion-free. 
(In order to easily guarantee that $\Lambda_1$ is large, we will choose $\gamma_1$ and $\gamma_2$ to be elements of a normal,
finite index subgroup of $\Gamma$ of index at least 3.)
We will see that it  is easy to iterate this construction to produce examples with 
arbitrarily large first betti number.

Our second construction involves amalgamated free products and produces examples with first betti number zero.
Let $\Delta=\langle \alpha,\beta\rangle$ be a malnormal, infinite index subgroup of $\Gamma$ freely generated
by $\alpha$ and $\beta$. 
Let $\phi:\Delta\to\Delta$ be an isomorphism such that the ratio of the
translation lengths of $\alpha$ and $\beta$ is different than the ratio of the translation lengths of $\phi(\alpha)$
and $\phi(\beta)$. 
We then construct 
$$\Lam_0=\Gamma *_\phi \Gamma$$
 from two copies of $\Gamma$ by identifying $\Del$ in the first copy with $\Del$ in the second copy via
the isomorphism $\phi$.
We argue, as before, that if $\Lam_0$ is linear, then there is
a representation $\rho$ of $\Lambda_0$ into $\GL_n(\mathbb R)$ such that the restriction of
$\rho$ to each factor determines an equivariant totally geodesic embedding of $X$ into $Y_n$.
It follows that the ratio of the translation lengths of $\alpha$ and $\beta$
agrees with the ratios of the translation lengths of $\phi(\alpha)$ and $\phi(\beta)$, which we have 
disallowed. (In order to  establish that $\Lambda_0$ is large we will also assume that $\Delta$ is
contained in a normal subgroup of $\Gamma$ of finite index at least 3.)

We regard the main advantage of our new constructions to be their relative simplicity and flexibility. For example, if one were given an explicit
presentation of a superrigid lattice, one could easily write down an explicit presentation of a group of the form $\Lambda_1$.

The first published examples of nonlinear hyperbolic groups are due to M.\ Kapovich \cite{MKap}. Gromov \cite{Gromov}
used small cancellation theory to
show that suitable quotients of a lattice $\Gam$ as above are infinite hyperbolic groups (see also \cite{champetier,Delzant,olshanskii}),
and then Kapovich used superrigidity results to show that any linear representation of these quotients has finite image.
In particular, these examples have Property (T), since they are quotients of Property (T) groups. It follows that these groups are
not large and hence are not abstractly commensurable with our examples.

The paper is organized as follows. In \S \ref{sec:Construction1}, we give the details of our constructions and show that our
groups have the claimed group-theoretic properties.
In \S \ref{sec:Superrigidity} we recall the necessary consequences of superrigidity for lattices in $\Sp(m,1)$, $m \ge 2$, or $\mathrm{F}_4^{(-20)}$. 
The proofs of nonlinearity are given in \S \ref{proofs}.

\medskip\noindent
{\bf Acknowledgments} The authors are grateful to David Fisher for conversations about superrigidity, to Daniel Groves, Jason Manning, and Henry Wilton for conversations about hyperbolic groups, and to Jack Button for comments on an early draft of the manuscript.
Canary and Tsouvalas were partially supported by NSF grants DMS-1306992 and DMS-1564362. Stover was partially supported by the National Science Foundation under Grant Number NSF 1361000 and Grant Number 523197 from the Simons Foundation/SFARI. The authors acknowledge support from U.S. National Science Foundation grants DMS 1107452, 1107263, 1107367 ``RNMS: GEometric structures And Representation varieties'' (the GEAR Network).

%%%%%%%%%%%%%%%%%%%%
\section{The constructions}\label{sec:Construction1}
%%%%%%%%%%%%%%%%%%%%

In this section, we give the details of the constructions described in the introduction and establish the group-theoretic
properties claimed there. Throughout this paper $G$ will be either $\Sp(m, 1)$ for $m \ge 2$ or $\mathrm{F}_4^{(-20)}$, so $G$ acts by isometries on a rank one symmetric space $X$, which is quaternionic hyperbolic \hbox{$m$-space} or the Cayley hyperbolic plane, respectively. Then $\Gamma$ will always denote a torsion-free cocompact lattice in $G$. In particular, $\Gamma$ is hyperbolic, admits a finite $K(\pi,1)$, 
$H^1(\Gamma,\mathbb R)=0$, and the cohomological dimension of
$\Gamma$ is the dimension of $X$.

\medskip

We first construct the examples with nontrivial first betti number.
If $n\ge 2$, let $\{\gamma_1,\ldots,\gamma_{2n}\}$  be primitive elements of $\Gamma$ with
distinct translation lengths.
The associated geodesics in $X/\Gamma$ are distinct, so
no nontrivial power of $\gamma_i$ is conjugate to a power of $\gamma_j$ for $i\ne j$.
We define
\[
\Lambda_n=\langle \Gamma,t_1,\ldots,t_n\ |\ t_i\gamma_{i}t_i^{-1}=\gamma_{i+n}\rangle.
\]
to be obtained by repeated HNN extensions.

In order to construct examples which are large and have betti number zero and one, we 
observe that $\Gamma$ contains a free, quasiconvex, malnormal subgroup $\Delta$ of rank two so that
$\Delta$ is contained in a finite index, normal subgroup $N$ of $\Gamma$ of index at least three.
We first note that, since $\Gamma$ is residually finite, it contains a finite index, normal subgroup $N$
of index at least three. I. Kapovich \cite[Thm 6.7]{kapovich-nonqc} showed that every non-elementary hyperbolic group 
contains a  malnormal quasiconvex subgroup  which is free of rank two. Let $F$ be a  free  malnormal quasiconvex subgroup of $\Gamma$ of rank two
and let $D$ be a subgroup of $F\cap N$ which is free of rank two. Since every finitely generated subgroup of  a free group is quasiconvex 
and $F$ is quasiconvex in $\Gamma$, we see that $D$  is quasiconvex in $\Gamma$. Kapovich's proof actually first constructs a free 
quasiconvex subgroup  of rank two and then shows that  this subgroup contains a free subgroup of rank two which is malnormal in the entire group. 
Therefore, $D$, and hence $N$, contains a  subgroup $\Delta$ which is free of rank two and malnormal and quasiconvex in $\Gamma$.

Let $\gamma_1$ and $\gamma_2$  be generators of $\Delta$ with distinct translation length.
Since $\Delta$ is malnormal in $\Gamma$, no nontrivial power of $\gamma_1$ is conjugate to a power of $\gamma_2$.
Let $\Lambda_1$ be the HNN extension of $\Gamma$ given by
$$\Lambda_1=\langle \Gamma,t\ |\ t\gamma_1t^{-1}=\gamma_2\rangle.$$
(If we do not require $\Lambda_1$ to be large, it would suffice to choose $\gamma_1$ and $\gamma_2$ to be
primitive elements with distinct translation length as in the construction of $\Lambda_n$ when $n\ge 2$.)

\medskip

We now construct the examples with trivial first betti number.
Let $\alpha$ and $\beta$ generate $\Delta$ and let 
$\phi:\Delta\to\Delta$ be an isomorphism such that the ratio of the
translation lengths of $\alpha$ and $\beta$ is different than the ratio of the translation lengths of $\phi(\alpha)$
and $\phi(\beta)$. 
We define
\[
\Lam_0 = \Gam *_{\phi} \Gam.
\]
to be obtained 
from two copies of $\Gamma$ by identifying $\Del$ in the first copy with $\Del$ in the second copy via
the isomorphism $\phi$. (If we do not require that $\Lam_0$ is large, it would suffice to choose $\Delta$ to be
the malnormal, quasiconvex subgroup of $\Gamma$ guaranteed by Kapovich \cite{kapovich-nonqc}.)

\begin{prop}
For all $n$, a group $\Lambda_n$ constructed as above is hyperbolic, torsion-free, large, one-ended, has
a finite $K(\pi,1)$, has first betti number $n$, and its cohomological dimension is the dimension of $X$. Moreover,
if $n\ge 1$,  $\Lambda_n$ admits a surjective homomorphism to the free group $F_n$ of rank $n$.
\end{prop}

\begin{proof}
That $\Lambda_n$ is torsion-free, one-ended, has
a finite $K(\pi,1)$, has first betti number $n$, and has cohomological dimension 
equal to the dimension of $X$ follows from standard facts about graphs of groups
(see, for example, Serre \cite[Chap.\ 1]{Serre} or Scott--Wall \cite{scott-wall}).
If $n\ge 1$, then $\Lambda_n$ clearly surjects onto the group freely generated by $\{t_1,\ldots, t_n\}$.
The fact that each $\Lambda_n$ is hyperbolic is a special case of the Bestvina--Feighn combination theorem \cite{BF}, 
which is explicitly stated in I.\ Kapovich \cite[Ex.\ 1.3]{IKap} as follows:

\begin{thm}{\ }
\begin{enumerate}
\item
If $A$ and $B$ are hyperbolic groups and $C$ is a quasiconvex subgroup of both $A$ and $B$ that is
malnormal in either $A$ or $B$, then $A*_CB$ is hyperbolic.
\item
If $A$ is a hyperbolic group and $a_1$ and $a_2$ are elements of $A$ so that no nontrivial power of $a_1$ is conjugate to a power of $a_2$, the HNN extension
\[
\langle A,t\ |\ t a_1t^{-1}=a_2\rangle
\]
is hyperbolic.
\end{enumerate}
\end{thm}

\noindent
Part (1) immediately implies that $\Lambda_0$ is hyperbolic, while part (2) gives that $\Lambda_n$ is hyperbolic
if $n\ge 1$. Also, notice that normal form for words in the HNN extension $\Lambda_{n-1}$ (see \cite[\S I.5]{Serre}) implies we still have that that no power of $\gamma_{n}$ is conjugate to a power of $\gamma_{2n}$ in $\Lambda_{n-1}$.

\medskip

We remarked above that $\Lambda_n$ is large for $n \ge 2$, so it remains to prove that $\Lambda_1$ and $\Lambda_0$ are also large.
%To complete the proof of the proposition, it remains to show that each $\Lambda_n$ is large.
%If $n\ge 2$, then $\Lambda_n$ clearly surjects onto the group freely generated by $t_1$ and $t_2$, so is large.
Suppose that $n=1$.  There is a surjective homomorphism
\[
p_1:\Lambda_1 \to H_1=\Gamma/N*\Z
\]
given by projecting $\Gamma$ onto $\Gamma/N$ and taking $t$ to the generator of $\Z$.
Let $J$ be a finite index subgroup of $H_1$ which is isomorphic to a free group of rank at least two, which exists, since $\Gamma/N$ has order at least three.
Then $p_1^{-1}(J)$ is a finite index subgroup of  $\Lam_1$ and
$p_1$ restricts to a surjection of $p_1^{-1}(J)$ onto $J$, so $\Lam_1$ is large.
%However, since $\pi(\gamma_2)$ is not a power of $\pi(\gamma_1)$ and $\pi(\gamma_2)$ is not a power of $\pi(\gamma_1)$, $H_1$ contains a finite index subgroup isomorphic to $F_r$ for some $r\ge 2$, so $\Lambda_1$ is large. In fact, the kernel of the obvious map from $H_1$ to $K_1\oplus K_2$ is free of rank at least two (see \cite[\S I.4.3]{Serre}).

\medskip

We now consider $\Lam_0$. There exists a surjective homomorphism
\[
p_0:\Lambda_0 \to H_0=\Gamma/N*\Gamma/N
\]
given by projecting the first factor of $\Lambda_0$ to the first factor of $H_0$ and the second factor of
$\Lambda_0$ to the second factor of $H_0$. Notice that this is well-defined since $\Delta$ has trivial image in both factors.
As above, $H_0$ contains a finite index subgroup which is isomorphic to a free group of rank at least two, so
$\Lam_0$ is large.
\end{proof}
 %%%%%%%%%%%%%%%%%%%%
 
%%%%%%%%%%%%%%%%%%%%
\noindent
{\bf Remarks:} 1) Kapovich \cite{kapovich-nonqc} further uses a malnormal quasiconvex free subgroup of a word hyperbolic group $G$ to 
construct a hyperbolic group $G^*$ which contains
$G$ as a non-quasiconvex subgroup. We note that $G^*$ is a quotient of a group of the form $\Lambda_2$,  obtained by identifying the
two stable letters,
so if $G$ is a superrigid rank one lattice then $G^*$ can be chosen to be nonlinear.

2) We expect that the techniques of Belegradek--Osin \cite{belegradek-osin}, which
also begin with quotients of superrigid lattices and employ more powerful small cancellation theoretic results,
also produce large, one-ended, nonlinear hyperbolic groups (in particular, see \cite[Thm.\ 3.1]{belegradek-osin}).

3) It is clear that one can construct infinitely many isomorphism classes of groups of the form $\Lambda_n$, for each $n$,
even if one begins with a fixed superrigid lattice $\Gamma$. For example, if $n\ge 1$, it follows readily from the JSJ
theory for hyperbolic groups, see Sela \cite{sela}, that the isomorphism type of  a group of the form 
$\Lambda_1$ is determined, up to finite ambiguity, by the conjugacy class of the pair $\{\gamma_1,\gamma_2\}$ in $\Gamma$.

%%%%%%%%%%%%%%%%%%%%

%%%%%%%%%%%%%%%%%%%%
\section{Superrigidity}\label{sec:Superrigidity}
%%%%%%%%%%%%%%%%%%%%

In this section, we record a version of the superrigidity theorem of Corlette \cite{Corlette} and 
Gromov--Schoen \cite{Gromov--Schoen} that is crafted for our purposes. In our statement
$Y_n$ will denote the symmetric space
\[
Y_n = Z \Or(n) \backslash \GL_n(\R) = \PO(n)\backslash\PGL_n(\R)
\]
associated with $\GL_n(\mathbb R)$, where $Z$ denotes the center of $\GL_n(\R)$.

%%%%%%%%%%%%%%%%%%%%
\begin{thm}\label{thm:Superrigidity1}
Suppose that $\Gam$ is a lattice in $G$, where $G$ is either $Sp(m,1)$ or $\mathrm{F}_4^{(-20)}$, $F$ is a field of characteristic zero, and $\rho : \Gam \to \GL_d(F)$ is a representation with infinite image.
\begin{enumerate}

\item There exists a faithful representation
\hbox{$\tau:\GL_d(F)\to \GL_n(\R)$} for some $n$ such that $\tau\circ\rho(\Gamma)$ has noncompact Zariski closure.

\item If $F=\mathbb R$ and $\rho(\Gamma)$ has noncompact Zariski closure in $\GL_d(\R)$, 
then there exists a $\rho$-equivariant totally geodesic map
\[
f_\rho : X \to Y_d,
\]
where $X=K\backslash G$ is the symmetric space associated with $G$.

\end{enumerate}
\end{thm}
%%%%%%%%%%%%%%%%%%%%

%%%%%%%%%%%%%%%%%%%%
\begin{pf}
Since $\Gamma$ is finitely generated we may assume that F is isomorphic to a subfield of $\mathbb{C}$. Moreover, $\GL_d(\C)$ is a subgroup of $\GL_{2d}(\R)$. It follows that there exists an injective representation \hbox{$\eta:\GL_d(F)\to \GL_n(\R)$} for some $n$, so we may assume that the original representation maps into $\GL_n(\R)$.

Fisher and Hitchman \cite[Thm.\ 3.7]{FH} then observe that the existing results on superrigidity imply that one
can factor $\rho$ as two representations
\[
\rho_i : \Gam \to \GL_{n_i}(\R) \subseteq \GL_n(\R)
\]
such that:
\begin{enumerate}

\item When $\rho_1$ is nontrivial, there is a group $G^\prime$ locally isomorphic to $G$, a continuous representation $\hat{\rho}_1 : G^\prime \to \GL_{n_1}(\R)$, and an embedding $\iota : \Gam \hookrightarrow G^\prime$ of $\Gam$ as a lattice in $G^\prime$ such that $\rho_1 = \hat{\rho}_1 \circ \iota$.

\item The image of $\rho_2$ is bounded, i.e., has compact Zariski closure.

\item The groups $\rho_1(\Gam)$ and $\rho_2(\Gam)$ commute and $\rho(\gam) = \rho_1(\gam) \rho_2(\gam)$ for all $\gam \in \Gam$.

\end{enumerate}
If $\rho_1$ is nontrivial, the continuous embedding $\hat{\rho}_1 : G^\prime \to \GL_{n_1}(\R)$ determines a totally geodesic embedding of $X$ into $Y_{n_1}$, hence into $Y_n$. Since $\rho_1$ and $\rho_2$ commute, this is a $\rho$-equivariant map.

\medskip

When $\rho_1$ is trivial, we follow arguments in the proof of \cite[Thm.\ 8.1]{MKap}. Note that our use of \cite[Thm.\ 3.7]{FH} allows us to know beforehand that the solvable radical considered in \cite{MKap} is trivial. 
As in \cite{MKap}, the fact that $\Gamma$ has Property (T) allows us to conclude that we may conjugate $\rho$
so that $\rho(\Gam) \subseteq \GL_n(k)$ for some number field $k$. Given an element $\sig \in \mathrm{Aut}(k / \Q)$, we can choose an extension of $\sig$ to an element of $\mathrm{Aut}(\C/\Q)$, which we continue to denote by $\sig$. Applying $\sig$ to matrix entries induces an embedding $\tau_\sigma : \GL_n(F) \to \GL_n(\C)$.

Following the adelic argument in \cite{MKap}, if $\rho(\Gamma)$ were bounded for every valuation of $k$ then $\rho(\Gamma)$ would be finite, which is a contradiction. Moreover, $\rho(\Gamma)$ must be bounded for every nonarchimedean valuation by nonarchimedean superrigidity \cite{Gromov--Schoen}. 
Consequently, there exists $\sig \in \mathrm{Aut}(k / \Q)$ such that 
$\tau_\sigma(\rho(\Gam))$ has noncompact Zariski closure in $\GL_n(\R)$ or $\GL_{2n}(\R)$, according to whether $\sig(k) \otimes_\sigma \R$ is $\R$ or $\C$.
Applying the previous argument to $\tau_\sigma\circ \rho$, there is a $(\tau_\sigma\circ\rho)$-equivariant totally geodesic embedding of $X$ into $Y_n$ or $Y_{2n}$, accordingly. This completes the sketch of the proof.
\end{pf}
%%%%%%%%%%%%%%%%%%%%

M.\ Kapovich \cite{MKap} also points out that superrigidity rules out faithful representations of $\Gamma$ into linear groups of fields of positive characteristic. 
Briefly, one shows that the image of $\rho$ lies in $\GL_n(k)$ where $k$ is a finite extension of $\mathbb{F}_p(x_1, \dots, x_n)$. Then, applying Gromov--Schoen superrigidity 
\cite{Gromov--Schoen} to each valuation of $k$ associated with some $x_i^{\pm 1}$, one sees that $\rho(\Gamma)$ is bounded in each field associated with such a valuation on $k$, as all valuations on $k$ are nonarchimedean. It follows that $\rho(\Gamma)$ is bounded and hence finite. Thus we have:

%%%%%%%%%%%%%%%%%%%%
\begin{prop}\label{lem:Nop}
If $\Gam$ is a lattice in either $Sp(m,1)$ or $\mathrm{F}_4^{(-20)}$ and $F$ is a field of characteristic $p>0$, then there does not exist a faithful representation of $\Gam$ into $\GL_n(F)$ for any $n$.
\end{prop}
%%%%%%%%%%%%%%%%%%%%

%%%%%%%%%%%%%%%%%%%%
\section{Proofs of nonlinearity}\label{proofs}
%%%%%%%%%%%%%%%%%%%%

\noindent To complete the proof of Theorem \ref{mainthm} it remains to prove:

%%%%%%%%%%%%%%%%%%%%
\begin{thm}
Groups of the form $\Lam_n$  constructed in Section \ref{sec:Construction1} are nonlinear.
\end{thm}
%%%%%%%%%%%%%%%%%%%%       

%%%%%%%%%%%%%%%%%%%%
\begin{pf}
We begin with a group of the form
\[
\Lambda_1=\langle \Gamma,t\ |\ t\gamma_1t^{-1}=\gamma_2\rangle
\]
constructed in Section \ref{sec:Construction1}, where $\Gamma$ is a cocompact lattice in $G$ and
$G$ is either $Sp(m,1)$ or $\mathrm{F}_4^{(-20)}$. Recall that $X$ is the symmetric space associated with $G$
and that $\gamma_1$ and $\gamma_2$ are assumed to have different translation lengths on $X$.

Suppose that $F$ is a field and $\eta:\Lam_1\to \GL_d(F)$ is a faithful representation. 
Applying Proposition \ref{lem:Nop} to the restriction $\rho=\eta|_\Gamma$ of $\eta$
to $\Gamma$, we conclude that $F$ has characteristic zero. Theorem \ref{thm:Superrigidity1} implies
that there exists a faithful representation \hbox{$\tau_\sigma:\GL_d(F)\to \GL_n(\R)$}, for some $n$
and a $(\tau_\sigma\circ\rho)$-equivariant totally geodesic embedding $f$ of $X$ into $Y_n$, where $Y_n$ is the symmetric space
associated with $\GL_n(\R)$.

Since $\tau_\sigma(\rho(\gamma_1))$ is conjugate to $\tau_\sigma(\rho(\gamma_2))$
in $\tau_\sigma(\eta(\Lambda_1))$, and hence in $\GL_n(\R)$, they have the same translation length on $Y_n$.
However, since $f$ is a ($\tau_\sigma\circ\rho$)-equivariant totally geodesic embedding, this implies that
$\gamma_1$ and $\gamma_2$ have the same translation length in $X$, which is a contradiction, hence $\Lam_1$ is nonlinear. Notice that if $n\ge 2$, then any group of the form $\Lam_n$ constructed  in Section \ref{sec:Construction1}
contains a subgroup of the form $\Lam_1$, so $\Lam_n$ is also nonlinear.

Now suppose we have a group of the form 
\[
\Lam_0=\langle \Gamma_1,\Gamma_2\ |\ \alpha_1=\phi(\alpha)_2,\ \beta_1=\phi(\beta)_2\rangle
\]
where each $\Gamma_i$ is a copy of $\Gamma$, $\Delta=\langle \alpha,\beta\rangle$ is a subgroup of $\Gamma$
freely generated by $\alpha$ and $\beta$,
$\Delta_i$ is the copy of $\Delta$ in $\Gamma_i$ and if $\delta\in\Delta$, then $\delta_i$ is the copy of
$\delta$ in $\Delta_i$. Moreover, $\phi$ is an automorphism of $\Delta$ so that the ratio of
the translation lengths of $\alpha$ and $\beta$ on $X$ differs from the ratio of translation lengths
of $\phi(\alpha)$ and $\phi(\beta)$ on $X$.

Suppose that $F$ is a field and $\eta:\Lam_0\to \GL_d(F)$ is a faithful representation. 
Let $\rho_1=\eta|_{\Gamma_1}$ and $\rho_2=\eta|_{\Gamma_2}$
We again apply Proposition \ref{lem:Nop} to conclude that $F$ has characteristic zero, Theorem \ref{thm:Superrigidity1} implies
that there exists a faithful representation \hbox{$\tau_\sigma:\GL_d(F)\to \GL_n(\R)$}, for some $n$
and a $(\tau_\sigma\circ\rho_1)$-equivariant embedding $f$ of $X$ into $Y_n$, where $Y_n$ is the symmetric space
associated with $\GL_n(\R)$.  Since $\tau_\sigma(\rho_1(\Delta_1))=\tau_\sigma(\rho_2(\Delta_2))$ has noncompact
Zariski closure, Theorem \ref{thm:Superrigidity1} implies
that there exists a 
$(\tau_\sigma\circ\rho_2)$-equivariant embedding $g$ of $X$ into $Y_n$. Notice that $\tau_\sigma(\rho_1(\alpha_1))=\tau_\sigma(\rho_2(\phi(\alpha)_2))$  
and that $\tau_\sigma(\rho_1(\beta_1))=\tau_\sigma(\rho_2(\phi(\beta)_2))$.

Since $f$ and $g$ are equivariant
totally geodesic embeddings, there exist positive constants $c_1$ and $c_2$ so that if $\gamma\in\Gamma$, 
then the ratio of the  translation length of $\tau_\sigma(\rho_i(\gamma_i))$ on $Y_n$ and the translation length of $\gamma$ on $X$ is $c_i$. Indeed, the metrics on $f(X)$ and $g(X)$ differ by a scalar multiple.
It follows that the ratio of the translation lengths of $\alpha$ and $\beta$ on $X$
agrees with the ratio of the translation lengths of $\phi(\alpha)$ and $\phi(\beta)$ on $X$.
However, this contradicts our assumptions, so $\Lam_0$ is nonlinear.
\end{pf}
%%%%%%%%%%%%%%%%%%%%

%%%%%%%%%%%%%%%%%%%%

%%%%%%%%%%%%%%%%%%%%
\end{document}